\documentclass[reqno,11pt]{amsart}
\usepackage{amsmath,mathrsfs,tikz}
\usepackage{enumerate}

\numberwithin{equation}{section}
\newtheorem{theorem}{Theorem}
\newtheorem{lemma}[theorem]{Lemma}

\newtheorem{proposition}[theorem]{Proposition}

\theoremstyle{definition}

\DeclareMathOperator{\length}{length}
\newcommand{\C}{\mathscr{C}}
\newcommand{\sz}{\tiny}

\title{Compositions colored by simplicial polytopic numbers}
\author{Daniel Birmajer}
\address{Department of Mathematics\\ Nazareth College\\ 4245 East Ave.\\ Rochester, NY 14618}
\author{Juan B. Gil}
\address{Penn State Altoona\\ 3000 Ivyside Park\\ Altoona, PA 16601}
\author{Michael D. Weiner}

\begin{document}
\maketitle

\begin{abstract}
For a given integer $d\ge 1$, we consider $\binom{n+d-1}{d}$-color compositions of a positive integer $\nu$ for which each part of size $n$ admits $\binom{n+d-1}{d}$ colors. We give explicit formulas for the enumeration of such compositions, generalizing existing results for $n$-color compositions (case $d=1$) and $\binom{n+1}{2}$-color compositions (case $d=2$). In addition, we give bijections from the set of $\binom{n+d-1}{d}$-color compositions of $\nu$ to the set of compositions of $(d+1)\nu - 1$ having only parts of size $1$ and $d+1$, the set of compositions of $(d+1)\nu$ having only parts of size congruent to $1$ modulo $d+1$, and the set of compositions of $(d+1)\nu + d$ having no parts of size less than $d+1$. Our results rely on basic properties of partial Bell polynomials and on a suitable adaptation of known bijections for $n$-color compositions.
\end{abstract}

%%%%%%%%%%%%%%%%%%%%%%%%%%%%%%%%%%%%%%%%%%%%
\section{Introduction}

A composition of a positive integer $\nu$ is an ordered $k$-tuple $(j_1, \dotsc, j_k)$, for $k \ge 1$, of positive integers called parts such that $j_1 + \dotsb + j_k = \nu$. We call $k$ the number of parts. 

Given a sequence of nonnegative integers $w=(w_n)_{n\in\mathbb{N}}$, we define a $w$-color composition of $\nu$ to be a composition of $\nu$ such that part $n$ can take on $w_n$ colors. If $w_{n} = 0$, it means that we do not use the integer $n$ in the composition. Such colored (weighted) compositions have been considered by many authors, starting as early as 1960 (maybe even earlier), and they continue to be of current research interest, see e.g.\ \cite{ABCM14,Eger15a,Eger15b,HM10}.

If we let $W_n$ be the number of $w$-color compositions of $n$, Moser and Whitney \cite{MW61} observed that the corresponding generating functions $w(t)=\sum_{n=1}^\infty w_n t^n$ and $W(t)=\sum_{n=1}^\infty W_n t^n$ satisfy the relation
\[ W(t) = \frac{w(t)}{1-w(t)}, \text{ or equivalently, }\; 1+W(t) = \frac{1}{1-w(t)}. \]
In other words, the sequence $(W_n)_{n\in\mathbb{N}}$ is the {\sc invert} transform of $(w_n)_{n\in\mathbb{N}}$, see \cite{BernsteinSloane}.

A refinement of this formula, considering summations over all weighted compositions of $n$ with exactly $k$ parts, was given by Hoggatt and Lind \cite{HL68}. They showed that the number of weighted compositions of $n$ with exactly $k$ parts is given by
\begin{equation*}
  c_{n,k}(w)=\sum_{\pi_k(n)} \frac{k!}{k_1!\cdots k_n!}\, w_1^{k_1}\cdots w_n^{k_n},
\end{equation*}
where the sum runs over all $k$-part partitions of $n$, that is, over all solutions of 
\[ k_1+2k_2+\cdots+nk_n=n \text{ such that } k_1+\cdots+k_n=k \text{ and } k_j\in\mathbb{N}_0 \text{ for all } j. \]
Note that $c_{n,k}(w)$ is precisely the partial (exponential) Bell polynomial $B_{n,k}(1!w_1, 2!w_2,\dots)$ multiplied by a factor $k!/n!$. Given the fact that the invert transform may be written in terms of partial Bell polynomials (via Fa{\`a} di Bruno's formula), the aforementioned result can be formulated as follows. 

\begin{theorem} \label{thm:coloredComp}
Let $w=(w_n)_{n\in\mathbb{N}}$ be a sequence in $\mathbb{N}_0=\mathbb{N}\cup\{0\}$. Then its invert transform \begin{equation*}
  W_n = \sum_{k=1}^n \frac{k!}{n!} B_{n, k} (1!w_1, 2! w_2, \dots) \;\text{ for } n\ge 1,
\end{equation*}
counts the number of $w$-color compositions of $n$. Moreover, $\frac{k!}{n!} B_{n, k} (1!w_1, 2! w_2, \dots)$ gives the number of such compositions made with exactly $k$ parts.
\end{theorem}

This way of looking at colored compositions was recently discussed in \cite{ABCM14} with a slightly different notation since they used the ordinary Bell polynomials instead of the exponential Bell polynomials considered here. In op.~cit., the authors used this viewpoint to revisit some known examples of colored and restricted compositions.

In \cite{Eger15b}, Eger used the fact that $B_{n, k} (1!w_1, 2! w_2, \dots) = \frac{n!}{k!} c_{n,k}(w)$ to derive identities for partial Bell polynomials from identites for weighted compositions. In Eger's work, the quantity $c_{n,k}(w)$ is denoted by $\binom{k}{n}_f$, where $f$ is the weight function $f(i)=w_i$.

\medskip
In this paper, we study families of integer compositions colored by the {\it simplicial polytopic numbers} given by the sequences 
\begin{equation*}
 p(d)=\left\{\tbinom{n+d-1}{d};\; n\in\mathbb{N}\right\} \text{ for } d\in\mathbb{N}. 
\end{equation*}

In Section~\ref{sec:enumeration}, we derive explicit formulas for the enumeration of $p(d)$-color compositions of $n$, denoted by $P_n(d)$, and for the sets of restricted compositions:
\begin{align*}
\C_{1,m}(n) &= \{\text{compositions of $n$ having only parts of size $1$ and $m$}\}, \\
\C_{\equiv 1(m)}(n) &= \{\text{compositions of $n$ having only parts of size $\equiv 1$ modulo $m$}\}, \\
\C_{\ge m}(n) &= \{\text{compositions of $n$ having no parts of size less than $m$}\},
\end{align*}
where $m$ is an arbitrary integer greater than 1. In particular, we obtain the identity
\begin{equation}\label{eq:equalsize}
 P_\nu(d) = \left|\C_{1,d+1}\big((d+1)\nu - 1\big)\right| = \left|\C_{\equiv 1(d+1)}\big((d+1)\nu\big)\right| 
 = \left|\C_{\ge d+1}\big((d+1)\nu + d\big)\right|,
\end{equation}
which implies that there is a one-to-one correspondence between the sets involved. 

Finally, in Section~\ref{sec:bijections}, we provide combinatorial proofs of the identity \eqref{eq:equalsize} by suitably modifying some of the bijections given by Shapcott \cite{Shapcott13}.

The results presented in this note provide a natural generalization of what is known for the set of $n$-color compositions (case $d=1$) and its bijections to the set of 1-2 compositions (denoted here by $\C_{1,2}$), the set of odd compositions ($\C_{\equiv 1(2)}$), and the set of 1-free compositions ($\C_{\ge 2}$).

%%%%%%%%%%%%%%%%%%%%%%%%%%%%%%%%%%%%%%%%%%%%
\section{Enumeration formulas}
\label{sec:enumeration}

As mentioned in the introduction, some properties of partial Bell polynomials can be formulated as properties of colored compositions and vice-versa. For instance, the known recurrence (cf.\ \cite[Eq.\ (3k), Sec.~3.3]{Comtet})
\begin{equation*}
 B_{n, k}(1! w_1, 2! w_2, \dots)  
 = \frac1{k} \sum_{j=1}^{n-k+1} \binom{n}{j} j!w_j\,  B_{n-j, k-1}(1! w_1, 2! w_2, \dots)
\end{equation*}
is equivalent to the identity
\begin{equation}\label{eq:Vandermonde}
 c_{n,k}(w) = \sum_{j= 1}^{n-k+1} w_j c_{n-j,k-1}(w)
\end{equation}
for colored compositions. In \cite{Eger15a}, this formula is referred to as Vandermonde convolution. 

The main contribution of this section is the following proposition on the enumeration of compositions colored by the simplicial polytopic numbers.

\begin{proposition} \label{prop:p(d)}
For $d\in\mathbb{N}$ let $p(d)=\big(p_n(d)\big)_{n\in\mathbb{N}}$ be the sequence of simplicial $d$-polytopic numbers $p_n(d) =\binom{n+d-1}{d}$. Then the number of $p(d)$-color compositions of $n$ is given by
\begin{equation*}
 P_n(d) =  \sum_{k=1}^n \binom{n+dk-1}{n-k},
\end{equation*}
and $\binom{n+dk-1}{n-k}$ gives the number of such compositions having exactly $k$ parts.
\end{proposition}

\begin{proof}
By Theorem~\ref{thm:coloredComp}, we just need to verify the identity
\begin{equation*}
  c_{n,k}(p(d)) = \frac{k!}{n!}B_{n, k}(1! p_1(d), 2! p_2(d), \dots) = \binom{n+dk-1}{n-k},
\end{equation*}
which we will prove by induction on $k$. For $k=1$ and all $n$ we have 
\begin{equation*}
  c_{n, 1}(p(d)) = \frac{1!}{n!} n!p_n(d) = \binom{n+d-1}{d} = \binom{n+d-1}{n-1}.
\end{equation*}
For $k>1$, we use \eqref{eq:Vandermonde} and the inductive step to get
\begin{align*}
 c_{n,k}(p(d))  
 &= \sum_{j= 1}^{n-k+1} p_j(d) c_{n-j,k-1}(p(d)) \\
 &= \sum_{j= 1}^{n-k+1} \binom{j+d-1}{j-1}  \binom{n-j+d(k-1)-1}{n-j-k+1} \\
 &= \sum_{j= 0}^{n-k} \binom{j+d}{j} \binom{n-j+d(k-1)-2}{n-k-j} \\
 &= (-1)^{n-k}\sum_{j= 0}^{n-k} \binom{-d-1}{j} \binom{-d(k-1)-k+1}{n-k-j} \\
 &= (-1)^{n-k} \binom{-dk-k}{n-k} = \binom{n + dk - 1}{n-k}.
\end{align*}
\end{proof}

As discussed in the introduction, in addition to the $p(d)$-color compositions, we are also interested in the sets of restricted compositions $\C_{1,m}(n)$, $\C_{\equiv 1(m)}(n)$, and $\C_{\ge m}(n)$. 
\begin{proposition} \label{prop:1_and_m}
Let $m>1$ be an integer. Then
\begin{equation*}
  \big|\C_{1,m}(n)\big| = \sum_{j=0}^{\lfloor \frac{n}{m}\rfloor} \binom{n-(m-1)j}{j} \;\text{ for } n\ge m.
\end{equation*}
\end{proposition}
\begin{proof}
Using Theorem~\ref{thm:coloredComp} with the sequence $(w_n)$ defined by $w_1=w_m=1$ and $w_j=0$ for $j\not=1,m$, we get

\begin{align*}
 |\C_{1,m}(n)| &= \sum_{k=1}^n \frac{k!}{n!}B_{n, k} (1!, 0, \dotsc, m!, 0 \dots) \\
 &=  \sum_{k=1}^n \sum_{\substack{k_1+k_m=k \\ k_1+mk_m=n}} \frac{k!}{k_1! k_m!} \\
 &=  \sum_{i+mj=n} \frac{(i+j)!}{i! j!} =  \sum_{j=0}^{\lfloor \frac{n}{m}\rfloor} \frac{(n-(m-1)j)!}{(n-mj)! j!}
 = \sum_{j=0}^{\lfloor \frac{n}{m}\rfloor} \binom{n-(m-1)j}{j}.
\end{align*}
\end{proof}

In the next two propositions, we will discuss formulas for $|\C_{\equiv 1(m)}(n)|$ and $|\C_{\ge m}(n)|$. While these formulas can be easily derived from Theorem~\ref{thm:coloredComp} and basic properties of the partial Bell polynomials, we will prove them using elementary facts about compositions. Recall that the number of compositions of $n$ with exactly $k$ parts is given by $\binom{n-1}{k-1}$. Moreover, the number of weak compositions\footnote{In a weak composition, parts are allowed to be 0.} of $n$ with $k$ parts is $\binom{n+k-1}{k-1}$.

\begin{proposition} \label{prop:1_mod_m}
Let $m>1$ be an integer. Then
\begin{equation*}
\big|\C_{\equiv 1(m)}(n)\big|=\sum_{j=0}^{\lfloor \frac{n}{m}\rfloor} \binom{n-(m-1)j-1}{j} \;\text{ for } n\ge m.
\end{equation*}
\end{proposition}
\begin{proof}
Let $(j_1, \dotsc, j_k)$ be a composition of $n$ with parts that are congruent to 1 modulo $m$. Then
\begin{equation*}
 n = j_1+\dots+j_k = (p_1m+1)+\dots+(p_km+1),
\end{equation*}
where $p_1,\dots,p_k\ge 0$. This is possible if and only if 
\begin{equation*}
 n-k\equiv 0 \hspace{-5pt}\pmod{m} \;\text{ and }\; \frac{n-k}{m} = p_1+\dots +p_k.
\end{equation*}
Writing $n=qm+r$ and $k=jm+r$ with $0\le r < m$, we get $\frac{n-k}{m}=q-j$, and so the number of compositions of $n$ with $k$ parts $\equiv 1\!\pmod{m}$ is the same as the number of weak compositions of $q-j$ with $k$ parts: $\binom{q-j+k-1}{k-1} = \binom{q-j+k-1}{q-j} = \binom{q-j+jm+r-1}{q-j}$. Thus
\begin{align*}
\big|\C_{\equiv 1(m)}(n)\big|
 &= \sum_{j=0}^q \binom{q-j+jm+r-1}{q-j} \\
 &= \sum_{j=0}^q \binom{j+(q-j)m+r-1}{j} = \sum_{j=0}^{\lfloor \frac{n}{m}\rfloor} \binom{n -(m-1)j-1}{j}.
\end{align*}
\end{proof}

\begin{proposition} \label{prop:greater=m}
Let $m>1$ be an integer. Then
\begin{equation*}
  \big|\C_{\ge m}(n)\big| = \sum_{k=1}^{\lfloor \frac{n-1}{m-1}\rfloor} \binom{n-(m-1)k-1}{k-1} \;\text{ for } n\ge m.
\end{equation*}
\end{proposition}
\begin{proof}
Let $(j_1, \dotsc, j_k)$ be a composition of $n$ with parts that are greater than or equal to $m$. Then
\begin{equation*}
 n = j_1+\dots+j_k = (i_1-1+m)+\dots+(i_k-1+m),
\end{equation*}
where $i_1,\dots,i_k\ge 1$. This is equivalent to the identity $n-(m-1)k=i_1+\dots+i_k$. Thus the number of compositions in $\C_{\ge m}(n)$ with $k$ parts is the same as the number of compositions of $n-(m-1)k$ into $k$ parts, which is $\binom{n-(m-1)k-1}{k-1}$.
\end{proof}

\medskip
As a consequence of propositions \ref{prop:p(d)}, \ref{prop:1_and_m}, \ref{prop:1_mod_m}, and \ref{prop:greater=m}, we get the following result.
\begin{theorem}
For every $d,\nu \in\mathbb{N}$, we have
\begin{equation*}
 P_\nu(d) = \left|\C_{1,d+1}\big((d+1)\nu - 1\big)\right| = \left|\C_{\equiv 1(d+1)}\big((d+1)\nu\big)\right| 
 = \left|\C_{\ge d+1}\big((d+1)\nu + d\big)\right|.
\end{equation*}
In other words, the set of $p(d)$-color compositions of $\nu$ is in one-to-one correspondence with the set of compositions of $(d+1)\nu - 1$ having only parts of size $1$ and $d+1$, the set of compositions of $(d+1)\nu$ having only parts of size congruent to $1$ modulo $d+1$, and the set of compositions of $(d+1)\nu + d$ having no parts of size less than $d+1$.
\end{theorem}

\medskip
We finish this section with an interesting observation made by the referee. Since 
\[ P_\nu(d) = \binom{\nu+dk-1}{\nu-k} = \binom{\nu-k+(d+1)k-1}{(d+1)k-1}, \] 
we can also conclude that there are as many $p(d)$-color compositions of $\nu$ with $k$ parts as there are uncolored weak compositions of $\nu-k$ with $(d+1)k$ parts. 

%%%%%%%%%%%%%%%%%%%%%%%%%%%%%%%%%%%%%%%%%%%%
\section{Combinatorial bijections}
\label{sec:bijections}

Based on bijections given by Shapcott \cite{Shapcott13} for $n$-color compositions, in this section, we will provide bijective maps between the set of $p(d)$-color compositions of $\nu$ and the sets $\C_{1,d+1}\big((d+1)\nu-1\big)$, $\C_{\equiv 1(d+1)}\big((d+1)\nu\big)$, and $\C_{\ge d+1}\big((d+1)\nu+d\big)$. 

\smallskip
For this purpose, we fix $d\in\mathbb{N}$ and consider the sets
\begin{align*}
  A_k(\nu) &=\{\text{$p(d)$-color compositions of $\nu$ with $k$ parts}\}, \\
  B_k(\nu) &=\{\text{binary strings of length $\nu+dk-1$ with exactly $(d+1)k-1$ ones}\}.
\end{align*}
\begin{proposition} \label{prop:binaryMap}
For any fixed $d$, there is a bijective map $T:A_k(\nu)\to B_k(\nu)$.
\end{proposition}

Before we prove this proposition, we need the following lemma.
\begin{lemma} \label{lem:binaryCount}
Let $d\in\mathbb{N}$ be fixed. For $n\in\mathbb{N}$, $n\ge d$, there is a bijection $\phi_d$ from $\left\{1,2,\dots,\binom{n}{d}\right\}$ to the set of binary words of length $n$ having exactly $d$ ones. 
\end{lemma}
\begin{proof}
Let $m\in\mathbb{N}$ be such that $m\le \binom{n}{d}$. Using the fact that $\binom{n}{d} = 1+\sum_{j=1}^{d} \binom{n-j}{d-j+1}$, we construct a binary word $w=\phi_d(m)$, having exactly $d$ ones, as follows:
\begin{enumerate}
\itemsep4pt
\item[$\circ$] Let $m_1=m-1$ and find $p_1$ such that $\binom{p_1}{d}\le m_1 < \binom{p_1+1}{d}$;
\item[$\circ$] for every $2\le j\le d$, let $m_j=m_{j-1}-\binom{p_{j-1}}{d-j+2}$ and find $p_{j}$ such that 
\[ \binom{p_{j}}{d-j+1}\le m_j < \binom{p_{j}+1}{d-j+1}; \]
\item[$\circ$] starting from the right, make the binary word $w$ of length $n$ having a 1 at each position $p_j+1$ for $j=1,\dots,d$, and adding as many zeros to the left as necessary.
\end{enumerate}
For example, for $n=8$, $d=3$, and $m=24$, we get $p_1=6$, $p_2=3$, and $p_3=0$, leading to the binary word $\phi_3(24)=01001001$.

Note that if $p_1+1,\dots,p_d+1$ are the 1-positions associated with $m$ and $q_1+1,\dots,q_d+1$ are the positions associated with $k$, then $m<k$ implies $\sum_{j=1}^d 2^{p_j}<\sum_{j=1}^d 2^{q_j}$. In other words, the corresponding words $\phi_d(m)$ and $\phi_d(k)$ are distinct, hence $\phi_d$ is one-to-one.
 
Given a binary word $w$ of length $n$, having $d$ ones, the inverse $m=\phi_d^{-1}(w)$ can be found as follows:
\begin{enumerate}
\itemsep4pt
\item[$\circ$] Label all of the characters of $w$ from right to left as $0,1,2,\dots,n-1$;
\item[$\circ$] label the 1's in $w$ from right to left as $1,2,\dots,d$ and record their positions as $p_1+1,\dots,p_d+1$ from left to right;
\item[$\circ$] define $m=1+\sum_{j=1}^d \binom{p_j}{d-j+1}$, with the convention that $\binom{a}{b}=0$ if $a<b$.  
\end{enumerate}
For example, for the binary word $w=01001001$, we get

\begin{center}
\begin{tikzpicture}
\draw node at (2,1) {%
\begin{tabular}{rcccccccc}
 &\sz 7 &\sz 6 &\sz 5 &\sz 4 &\sz 3 &\sz 2 &\sz 1 &\sz 0 \\[-2pt]
 $w =$\!\! &\bf 0 &\bf 1 &\bf 0 &\bf 0 &\bf 1 &\bf 0 &\bf 0 &\bf 1 \\[-3pt]
 & &\sz 3 & & &\sz 2 & & &\sz 1 \\
\end{tabular}};%
\draw[rounded corners=2mm, very thin] (0.8,0.35) rectangle (1.25,1.6);
\draw[rounded corners=2mm, very thin] (2.5,0.35) rectangle (2.95,1.6);
\draw[rounded corners=2mm, very thin] (4.23,0.35) rectangle (4.68,1.6);
\end{tikzpicture}
\end{center}
and therefore $\phi_3^{-1}(01001001)=1+\binom{6}{3}+\binom{3}{2}+\binom{0}{1} = 1+20+3+0 = 24$.
\end{proof}

\medskip
With the help of $\phi_d$ and $\phi_d^{-1}$ we now proceed to prove the above proposition.

\begin{proof}[Proof of Proposition~\ref{prop:binaryMap}]
Let $\alpha=(n_1^{c_1},\dots,n_k^{c_k})$ be an element of $A_k(\nu)$, where each $n_i^{c_i}$ is a part of size $n_i$ with color $1\le c_i\le \binom{n_i+d-1}{d}$. For every part $n_i^{c_i}$ let $w_i=\phi_d(c_i)$ be the binary word of length $n_i+d-1$ obtained through the algorithm from Lemma~\ref{lem:binaryCount}. We then concatenate the $k$ binary words associated with each part of the composition $\alpha$, adding an extra 1 between consecutive parts, to create a binary string $\beta=T(\alpha)$ of length
\begin{equation*}
 (k-1) + \sum_{i=1}^k (n_i+d-1) =  k-1+\nu+(d-1)k = \nu+dk-1
\end{equation*}
with exactly $(k-1)+dk = (d+1)k-1$ ones. In other words, $\beta$ is an element of $B_k(\nu)$. 

For example, for $d=2$ each part $n$ takes on $d_n(2)=\binom{n+1}{2}$ colors, so for $n=1,2,3$ the above map $T(\alpha)=\beta$ gives the following:

\begin{center}
{\small
\begin{tabular}{@{\hspace{1em}}c |@{\hspace{1em}}c}
$d_n(2)$-color comp. & binary word \\ \hline\hline
& \\[-8pt]
$(1_1) $& 11
\end{tabular}
\vspace{1em}

\begin{tabular}{@{\hspace{1em}}c |@{\hspace{1em}}c}
$d_n(2)$-color comp. & binary word \\ \hline\hline
& \\[-8pt]
$(2_1)$ & 011 \\
$(2_2)$ & 101 \\
$(2_3)$ & 110 \\[3pt]
$(1_1,1_1)$ & 11{\color{red}1}11
\end{tabular}
\vspace{1em}

\begin{tabular}{@{\hspace{1em}}c |@{\hspace{1em}}c}
$d_n(2)$-color comp. & binary word \\ \hline\hline
& \\[-8pt]
$(3_1)$ & 0011 \\
$(3_2)$ & 0101 \\
$(3_3)$ & 0110 \\
$(3_4)$ & 1001 \\
$(3_5)$ & 1010 \\
$(3_6)$ & 1100 \\[3pt]
$(2_1,1_1)$ & 011{\color{red}1}11 \\
$(2_2,1_1)$ & 101{\color{red}1}11 \\
$(2_3,1_1)$ & 110{\color{red}1}11 \\[3pt]
$(1_1,2_1)$ & 11{\color{red}1}011 \\
$(1_1,2_2)$ & 11{\color{red}1}101 \\
$(1_1,2_3)$ & 11{\color{red}1}110 \\[3pt]
$(1_1,1_1,1_1)$ & 11{\color{red}1}11{\color{red}1}11
\end{tabular}
}
\end{center}

\smallskip
The above map $T$ is reversible. Given a binary string $\beta$ in $B_k(\nu)$, we create a composition in $A_k(\nu)$ by means of the following inductive algorithm:
\begin{enumerate}
\itemsep1pt
\item[$\circ$] Write $\beta$ as $w_1 1 \beta'$, where $w_1$ is a binary string with exactly $d$ ones;
\item[$\circ$] let $c_1=\phi_d^{-1}(w_1)$ and let $n_1^{c_1}$ be the part of size $n_1=\length(w_1)$ with color $c_1$;
\item[$\circ$] remove the one after $w_1$ and repeat the algorithm with $\beta'$ until it has only $d$ ones.
\end{enumerate}
Since every such $\beta\in B_k(\nu)$ has exactly $(d+1)k-1$ ones, the above algorithm will create $k$ parts with $n_1+\dots+n_k=\nu$, leading to a composition $\alpha=(n_1^{c_1},\dots,n_k^{c_k})$ in $A_k(\nu)$ such that $T(\alpha)=\beta$.
\end{proof}

Let $\mathscr{A}_{p(d)}(\nu)=\bigcup_{k=1}^{\nu} A_k(\nu)$ be the set of $p(d)$-color compositions of $\nu$.

\subsection*{Map $\mathscr{A}_{p(d)}(\nu) \to \C_{1,d+1}\big((d+1)\nu-1\big)$}
For $\alpha=(n_1^{c_1},\dots,n_k^{c_k})$ in $\mathscr{A}_{p(d)}(\nu)$, let $\beta=T(\alpha)$ be the binary string of length $\nu+dk-1$ from Proposition~\ref{prop:binaryMap}, having exactly $(d+1)k-1$ ones and $\nu-k$ zeros. If we treat every character 1 in $\beta$ as a separate part and replace every 0 by $d+1$, we get a unique composition of $(d+1)k-1 + (\nu-k)(d+1) = (d+1)\nu-1$, having only parts of size 1 and $d+1$. This map is clearly a bijection. 

\enlargethispage{1cm}
For example, for $d=2$ and $\nu=3$, we get

\medskip
\begin{center}
{\small
\begin{tabular}{@{\hspace{1em}}c |@{\hspace{1em}}c |@{\hspace{1em}}c}
$d_n(2)$-color comp. & binary word & comp.\ in $\C_{1,3}(8)$ \\ \hline\hline
&& \\[-8pt]
$(3_1)$ & 0011 & (3,3,1,1)\\
$(3_2)$ & 0101 & (3,1,3,1)\\
$(3_3)$ & 0110 & (3,1,1,3)\\
$(3_4)$ & 1001 & (1,3,3,1)\\
$(3_5)$ & 1010 & (1,3,1,3)\\
$(3_6)$ & 1100 & (3,3,1,1)\\[3pt]
$(2_1,1_1)$ & 011111 & (3,1,1,1,1,1)\\
$(2_2,1_1)$ & 101111 & (1,3,1,1,1,1)\\
$(2_3,1_1)$ & 110111 & (1,1,3,1,1,1)\\[3pt]
$(1_1,2_1)$ & 111011 & (1,1,1,3,1,1)\\
$(1_1,2_2)$ & 111101 & (1,1,1,1,3,1)\\
$(1_1,2_3)$ & 111110 & (1,1,1,1,1,3)\\[3pt]
$(1_1,1_1,1_1)$ & 11111111 &  (1,1,1,1,1,1,1,1)
\end{tabular}
}
\end{center}

\subsection*{Map $\mathscr{A}_{p(d)}(\nu) \to \C_{\equiv 1(d+1)}\big((d+1)\nu\big)$}
For $\alpha=(n_1^{c_1},\dots,n_k^{c_k})$ in $\mathscr{A}_{p(d)}(\nu)$, let $\beta=T(\alpha)$ be the binary string of length $\nu+dk-1$ from Proposition~\ref{prop:binaryMap}, having exactly $(d+1)k-1$ ones and $\nu-k$ zeros. 
Using the 1's in $\beta$ as separators, we now construct a composition as follows: To the left and right of every 1 in the binary string $\beta$, replace a string of $j$ zeros with a string of $(d+1)j+1$ zeros, which then represents a part of size $(d+1)j+1$. Since there are $(d+1)k-1$ separators, the constructed composition has $(d+1)k$ parts and the new total number of zeros is $(d+1)(\nu-k)+(d+1)k=(d+1)\nu$.

In conclusion, the above (clearly reversible) construction gives a composition of $(d+1)\nu$ having only parts of size congruent to $1$ modulo $d+1$.

For example, for $d=2$ and $\nu=3$, we get

\smallskip
\begin{center}
{\small
\begin{tabular}{@{\hspace{1em}}c |@{\hspace{1em}}c |@{\hspace{1em}}c |@{\hspace{1em}}c}
$d_n(2)$-color comp. & binary word & zeros as parts & comp.\ in $\C_{\equiv 1(3)}(9)$ \\ \hline\hline
&&& \\[-8pt]
$(3_1)$ & 0011 & 00000001010 & (7,1,1)\\
$(3_2)$ & 0101 & 00001000010 & (4,4,1)\\
$(3_3)$ & 0110 & 00001010000 & (4,1,4)\\
$(3_4)$ & 1001 & 01000000010 & (1,7,1)\\
$(3_5)$ & 1010 & 01000010000 & (1,4,4)\\
$(3_6)$ & 1100 & 01010000000 & (1,1,7)\\[3pt]
$(2_1,1_1)$ & 011111 & 00001010101010 & (4,1,1,1,1,1) \\
$(2_2,1_1)$ & 101111 & 01000010101010 & (1,4,1,1,1,1)\\
$(2_3,1_1)$ & 110111 & 01010000101010 & (1,1,4,1,1,1)\\[3pt]
$(1_1,2_1)$ & 111011 & 01010100001010 & (1,1,1,4,1,1)\\
$(1_1,2_2)$ & 111101 & 01010101000010 & (1,1,1,1,4,1)\\
$(1_1,2_3)$ & 111110 & 01010101010000 & (1,1,1,1,1,4)\\[3pt]
$(1_1,1_1,1_1)$ & 11111111 & 01010101010101010 &  (1,1,1,1,1,1,1,1,1)
\end{tabular}
}
\end{center}

\subsection*{Map $\mathscr{A}_{p(d)}(\nu) \to \C_{\ge d+1}\big((d+1)\nu+d\big)$}
For $\alpha=(n_1^{c_1},\dots,n_k^{c_k})$ in $\mathscr{A}_{p(d)}(\nu)$, let $\beta=T(\alpha)$ be the binary string of length $\nu+dk-1$ from Proposition~\ref{prop:binaryMap}, having exactly $(d+1)k-1$ ones and $\nu-k$ zeros. 
Using now the 0's in $\beta$ as separators, we construct a composition as follows: To the left and right of every 0 in the binary string $\beta$, replace a string of $j$ ones with a string of $j+d+1$ ones, which then represents a part of size $j+d+1$. Since there are $\nu-k$ separators, the constructed composition has $\nu-k+1$ parts and the new total number of ones is $(d+1)k-1+(d+1)(\nu-k+1) = (d+1)\nu+d$.

Thus the above reversible construction gives a composition of $(d+1)\nu+d$ having no parts of size less than $d+1$.

For example, for $d=2$ and $\nu=3$, we get

\smallskip
\begin{center}
{\small
\begin{tabular}{@{\hspace{1em}}c |@{\hspace{1em}}c |@{\hspace{1em}}c |@{\hspace{1em}}c}
$d_n(2)$-color comp. & binary word & ones as parts & comp.\ in $\C_{\ge 3}(11)$ \\ \hline\hline
&&& \\[-8pt]
$(3_1)$ & 0011 & 1110111011111 & (3,3,5)\\
$(3_2)$ & 0101 & 1110111101111 & (3,4,4)\\
$(3_3)$ & 0110 & 1110111110111 & (3,5,3)\\
$(3_4)$ & 1001 & 1111011101111 & (4,3,4)\\
$(3_5)$ & 1010 & 1111011110111 & (4,4,3)\\
$(3_6)$ & 1100 & 1111101110111 & (5,3,3)\\[3pt]
$(2_1,1_1)$ & 011111 & 111011111111 & (3,8) \\
$(2_2,1_1)$ & 101111 & 111101111111 & (4,7)\\
$(2_3,1_1)$ & 110111 & 111110111111 & (5,6)\\[3pt]
$(1_1,2_1)$ & 111011 & 111111011111 & (6,5)\\
$(1_1,2_2)$ & 111101 & 111111101111 & (7,4)\\
$(1_1,2_3)$ & 111110 & 111111110111 & (8,3)\\[3pt]
$(1_1,1_1,1_1)$ & 11111111 & 11111111111 & (11)
\end{tabular}
}
\end{center}

\medskip
\subsection*{Acknowledgement}
We would like to thank the referee for the thorough report and the very useful comments that we used to improve Section~\ref{sec:enumeration}.

%%%%%%%%%%%%%%%%%%%%%%%%%%%%%%%%%%%%%%%%%%%%

\end{document}